\documentclass[10pt]{article}  

\usepackage{amsmath, amsthm, amsfonts, amssymb}
\usepackage[bookmarks]{hyperref}
\usepackage{indentfirst} 
\usepackage{marvosym}
\usepackage{enumitem}
\setenumerate[1]{leftmargin = 0.6cm}

\usepackage[a4paper]{geometry}

\newtheorem{lemma}{Lemma}[section]
\newtheorem{theorem}[lemma]{Theorem}
\newtheorem{proposition}[lemma]{Proposition}
\newtheorem{corollary}[lemma]{Corollary}
\renewenvironment{proof}[1][\proofname]{{\noindent\bf #1. }}{\qed}
\newtheorem{theoremletters}{Theorem}

\newtheorem{conjectureletters}[theoremletters]{Conjecture}

\theoremstyle{definition}
\newtheorem{example}[lemma]{Example}
\newtheorem{remark}[lemma]{Remark}

\newcommand{\abs}[1]{\ensuremath{|#1|}}
\newcommand{\op}{\operatorname}
\newcommand{\ce}[2]{\pmb{\op{C}}_{#1}(#2)}

\newcommand{\ze}[1]{\pmb{\op{Z}}(#1)}

\newcommand{\rad}[2]{\pmb{\op{O}}_{#1}(#2)}
\newcommand{\syl}[2]{\op{Syl}_{#1}\hspace*{-0.7mm}\left(#2\right)}
\newcommand{\hall}[2]{\op{Hall}_{#1}\hspace*{-0.7mm}\left(#2\right)}

\newcommand{\conp}{\op{Con}(G_{p'})}

\usepackage{tikz}

\begin{document}

\title{\bf Groups whose common divisor graph on $p$-regular classes has diameter three}

\author{\sc  M.J. Felipe $^{*}$ $\cdot$ M.K. Jean-Philippe $^{\diamond}$ $\cdot$ V. Sotomayor
\thanks{Instituto Universitario de Matemática Pura y Aplicada (IUMPA-UPV), Universitat Polit\`ecnica de Val\`encia, Camino de Vera s/n, 46022 Valencia, Spain. \newline \Letter: \texttt{mfelipe@mat.upv.es, vsotomayor@mat.upv.es}\newline 
ORCID: 0000-0002-6699-3135, 0000-0001-8649-5742  \newline
\indent $^{\diamond}$Departamento de Matemáticas, Instituto Tecnológico de Santo Domingo (INTEC) and Universidad Autónoma de Santo Domingo (UASD), Santo Domingo, Dominican Republic. \newline \Letter: \texttt{1097587@est.intec.edu.do; mjean46@uasd.edu.do}
\newline \rule{6cm}{0.1mm}\newline
The first and last authors are supported by Ayuda a Primeros Proyectos de Investigación (PAID-06-23) from Vice\-rrectorado de Investigación de la Universitat Politècnica de València (UPV), and by Proyecto CIAICO/2021/163 from Generalitat Valenciana (Spain). The results in this paper are part of the second author's Ph.D. thesis.\newline
}}

\date{}

\maketitle

\begin{abstract}
\noindent Let $G$ be a finite $p$-separable group, for some fixed prime $p$. Let $\Gamma_p(G)$ be the common divisor graph built on the set of non-central conjugacy classes of $p$-regular elements of $G$: this is the graph whose vertices are the conjugacy classes of those non-central elements of $G$ such that $p$ does not divide their orders, and two distinct vertices are adjacent if and only if the greatest common divisor of their lengths is strictly greater than one. The aim of this paper is twofold: to positively answer an open question concerning the maximum possible distance in $\Gamma_p(G)$ between a vertex with maximal cardinality and any other vertex, and to study the $p$-structure of $G$ when $\Gamma_p(G)$ has diameter three.

\medskip

\noindent \textbf{Keywords} Finite groups $\cdot$ Conjugacy classes $\cdot$ $p$-regular elements $\cdot$ Common divisor graph

\smallskip

\noindent \textbf{2020 MSC} 20E45 $\cdot$ 20D20 
\end{abstract}


\section{Introduction}

Only finite groups will be considered in this paper. The interplay between the algebraic structure of a group and the arithmetical features of the lengths of its conjugacy classes is a widely-investigated research area within the theory of finite groups. This general issue has attracted the interest of many authors, who have investigated several variations on the theme over the past few decades. For instance, the information provided by the lengths of conjugacy classes of $p$\emph{-regular elements} of a group $G$ (i.e. elements with order not divisible by $p$, for a fixed prime $p$) is deeply related to its local structure. Let us denote by $G_{p'}$ the set of $p$-regular elements of $G$, and by $\operatorname{Con}(G_{p'})=\{x^G \; : \; x\in G_{p'}\}$ the set of conjugacy classes of $p$-regular elements of $G$. A useful tool to capture the arithmetical properties associated with the sizes of these conjugacy classes is the \emph{common divisor graph} $\Gamma_p(G)$: this is the (simple undirected) graph whose vertices are the non-central classes of $\operatorname{Con}(G_{p'})$, and two distinct vertices $B$ and $C$ are adjacent whenever their cardinalities have a common prime divisor, i.e. whenever $\operatorname{gcd}(|B|, |C|)>1$. This graph was originally defined by A. Beltrán and M.J. Felipe in \cite{BF1} and, when $p$ does not divide the order of $G$, it generalizes the graph $\Gamma(G)$ formerly introduced by E.A. Bertram, M. Herzog and A. Mann in \cite{BHM} that considers all the non-central conjugacy classes of $G$. Two main questions that naturally arise in this setting are: what can be said about the ($p$-)structure of $G$ if some information on $\Gamma(G)$ ($\Gamma_p(G)$) is known, and which graphs can occur as $\Gamma(G)$ ($\Gamma_p(G)$) for some group $G$? We refer the interested reader to Section 5 of the excellent survey \cite{CC} due to A.R. Camina and R.D. Camina for an overview of results on this topic.

Regarding the above second question, it was proved in \cite{BHM} that $\Gamma(G)$ has at most two connected components, and its diameter is at most four in the connected case. Later on the upper bound on the diameter of $\Gamma(G)$ was improved to three, which is best possible (cf. \cite{CHM}). These same bounds also hold for the graph $\Gamma_p(G)$ when $G$ is $p$-separable, as it was shown in \cite{BF1}. Note that the graph $\Gamma_p(G)$ can be viewed as the subgraph of $\Gamma(G)$ induced by those vertices of $\Gamma(G)$ whose representatives are $p$-regular elements. Despite of this fact, the upper bounds for the number of connected components and the diameter (in the connected case) of $\Gamma_p(G)$ cannot be directly derived from those of $\Gamma(G)$; indeed, the loss of information about conjugacy classes of elements with order divisible by $p$ leads to the development of new techniques for the graph $\Gamma_p(G)$ and, even so, some situations are still not entirely understood. For instance, the structure of $G$ when $\Gamma(G)$ is non-connected was characterised in \cite{K} (and later in \cite{BHM}, see Theorem~\ref{disconnected} below), whilst the corresponding situation for the graph $\Gamma_p(G)$, in spite of the in-depth analysis carried out by Beltrán and Felipe in \cite{BF3} (see Theorem~\ref{disconnected-p-reg}), is a problem that still remains unsolved.

The previously mentioned upper bound for the diameter of $\Gamma_p(G)$, when it is connected, was mainly based on the fact that the maximum possible distance between a conjugacy class in $\conp$ of maximal length and any other class is at most 3 (cf. \cite[Proposition 1]{BF1}); nonetheless, it was an open question whether this bound was sharp. Our first main result demonstrates that this maximum possible distance is actually 2, which also happens in the ordinary graph $\Gamma(G)$ (cf. \cite[Corollary 1]{BF1}). As a consequence, it can be reproved in an alternative way that the diameter of $\Gamma_p(G)$ is at most 3 (see Corollary~\ref{cor_diam}). Let us denote by $d(B,C)$ the distance between two non-central classes $B$ and $C$.

\begin{theoremletters}
\label{teoA}
Let $G$ be a $p$-separable group, for a given prime $p$, and suppose that $\Gamma_p(G)$ is connected. If $B_0$ is a non-central conjugacy class in $\operatorname{Con}(G_{p'})$ of maximal length, then $$d(B_0, D)\leq 2,$$ for every non-central conjugacy class $D\in \operatorname{Con}(G_{p'})$.
\end{theoremletters}

We highlight that a key ingredient in the proof of the above result is \cite[Corollary C]{FMS} (see Theorem~\ref{clave} below), which particularly asserts that all elements in $\operatorname{Con}(G_{p'})$ cannot have lengths either $p$-numbers or $p'$-numbers, simultaneously.

Next we put focus on the aforesaid question concerning how is the $p$-structure of $G$ influenced by the theoretical structure of the graph $\Gamma_p(G)$. We have previously mentioned that the diameter of $\Gamma_p(G)$ is at most three, so it naturally arises the study of the $p$-structure of groups $G$ such that the diameter of $\Gamma_p(G)$ is exactly three. Precisely, our second main result partially addresses it.

Let us denote by $\pi(g^G)$ the set of prime divisors of $|g^G|=|G:\ce{G}{g}|$. Recall that $G$ is called \emph{quasi-Frobenius} if $G/\ze{G}$ is Frobenius, and the inverse images in $G$ of the kernel and a complement of $G/\ze{G}$ are then called the kernel and a complement of $G$.

\begin{theoremletters}
\label{teoB}
Let $G$ be a $p$-separable group, for a given prime $p$. Suppose that $d(x^G, y^G)=3$, where $x^G, y^G\in\conp$ are non-central. Set $\pi_x=\pi(x^G)$, $\pi_y=\pi(y^G)$ and $\pi=\pi_x\cup\pi_y$. If $p$ divides $\operatorname{min}(|x^G|,|y^G|)$, then $G$ is $p$-nilpotent, and every $p$-complement $H$ of $G$ verifies that $H=\rad{\pi}{H} \times \rad{\pi'}{H}$ where $\rad{\pi}{H}$ is a quasi-Frobenius group with abelian kernel and complements. Moreover, each complement of $H$ is centralised by some Sylow $p$-subgroup of $G$, and $\ze{H}=H\cap\ze{G}$.
\end{theoremletters}

It is worth mentioning that the group may not be $p$-nilpotent whenever either $p$ divides $\max(|x^G|,|y^G|)$ or $p\notin \pi$, as we show in Example~\ref{example}. Nevertheless the structure of the $p$-complements of $G$ seems to be the same, as we claim in the next more general conjecture. This conjecture is also motivated by the main result of \cite{K}, where L.S. Kazarin characterised the structure of groups $G$ such that the graph $\Gamma(G)$ for ordinary classes possesses two \emph{isolated vertices}, i.e. two non-central classes $x^G$ and $y^G$ such that any other class has coprime length with one of them.

\begin{conjectureletters}
\label{conC}
Let $G$ be a $p$-separable group, for a given prime $p$. Suppose that $x^G,y^G\in\conp$ are non-central classes such that any other class $z^G\in\conp$ has coprime length with either $x^G$ or $y^G$. Set $\pi_x=\pi(x^G)$, $\pi_y=\pi(y^G)$ and $\pi=\pi_x\cup\pi_y$. Then every $p$-complement $H$ of $G$ verifies that $H=\rad{\pi}{H} \times \rad{\pi'}{H}$ where $\rad{\pi}{H}$ is a quasi-Frobenius group with abelian kernel and complements.
\end{conjectureletters}

Note that the feature of $\Gamma_p(G)$ in Conjecture~\ref{conC} particularly means that either it has dia\-meter three, or it is non-connected. In the former case, Theorem~\ref{teoB} provides some evidence for its veracity, although the situations where $p$ divides $\operatorname{max}(|x^G|,|y^G|)$ or $p\notin \pi$ are still open. In the last case, as previously mentioned, the $p$-structure of $G$ is still not well understood since there is a situation not solved yet: when the lengths of all classes lying in the connected component that contains a class in $\conp$ of maximal length are divisible by at most two primes $p$ and $q$, and a Sylow $q$-subgroup is non-abelian (see Theorem~\ref{disconnected-p-reg} below). In fact, there is a strong connection between both cases, as we further comment in Remark~\ref{remark}. Consequently, Conjecture~\ref{conC} seems to be a challenging problem.


\section{Preliminaries}

In the sequel, if $n$ is a positive integer, then we write $\pi(n)$ for the set of prime divisors of $n$. For the set of prime divisors of $|G|$ we simply write $\pi(G)$ and, as mentioned before, if $B=b^G$ is a conjugacy class of $G$, then $\pi(B)$ is the set of prime divisors of $|B|=|G:\ce{G}{b}|$. As usual, given a prime $p$, the set of all Sylow $p$-subgroups of $G$ is denoted by $\syl{p}{G}$, and $\hall{\pi}{G}$ is the set of all Hall $\pi$-subgroups of $G$ for a set of primes $\pi$. We write $\pi'$ for the set of primes that do not belong to $\pi$. The order of an element $g\in G$ is $o(g)$, and for any subset $\sigma\subseteq \pi(o(g))$, the $\sigma$-part of $g$ is denoted by $g_{\sigma}$. Finally, $n_{\pi}$ denotes the largest $\pi$-number that divides the positive integer $n$. The remaining notation and terminology used are standard in the frameworks of group theory and graph theory.

The following well-known lemma is used frequently and without further reference.

\begin{lemma}
Let $N$ be a normal subgroup of a group $G$, and let $\pi$ be a set of primes. Then:
\vspace*{-2mm}
\begin{itemize}
\setlength{\itemsep}{-1mm}
\item[\emph{(a)}] $|x^N|$ divides $|x^G|$, for any $x\in N$.
\item[\emph{(b)}] $|(xN)^{G/N}|$ divides $|x^G|$, for any $x\in G$.
\item[\emph{(c)}] If $xN\in G/N$ is a $\pi$-element, then $xN=yN$ for some $\pi$-element $y\in G$.
\item[\emph{(d)}] If $x,y\in G$ have coprime orders and they commute, then $\ce{G}{xy}=\ce{G}{x}\cap\ce{G}{y}$. In particular, both $|x^G|$ and $|y^G|$ divide $|(xy)^G|$.
\end{itemize}
\end{lemma}

The situation when $\Gamma(G)$, the common divisor graph for ordinary classes, is non-connected was completely characterised in \cite{K} (and also in \cite{BHM} using different techniques), whilst the corresponding situation for $\Gamma_p(G)$ was deeply studied in \cite{BF3}.

\begin{theorem}
\label{disconnected}
Let $G$ be a group. Then $\Gamma(G)$ is non-connected if and only if $G$ is a quasi-Frobenius group with abelian kernel and complements. Further, $G$ has exactly two non-trivial class sizes, namely the orders of the Frobenius kernel and the Frobenius complements of $G/\ze{G}$.
\end{theorem}

\begin{proof}
The first claim is \cite[Theorem]{K}. The second assertion is straightforward.
\end{proof}

\medskip

\begin{theorem}
\label{disconnected-p-reg}
Let $G$ be a $p$-separable group, for a given prime $p$, and let $H$ be a $p$-complement of $G$. Suppose that $\Gamma_p(G)$ is non-connected. Let $B_0\in\conp$ be a non-central class of maximal cardinality, and let $\pi_0$ be the set of prime divisors of the lengths of all conjugacy classes lying in the same connected component as $B_0$.
\begin{itemize}
\setlength{\itemsep}{-1mm}
\item[\emph{(a)}] If $p\notin \pi_0$, then $G$ is $p$-nilpotent, $H$ is quasi-Frobenius with abelian kernel and complements, and $\ze{H}=H\cap\ze{G}$. Moreover, each complement of $H$ is centralised by some Sylow $p$-subgroup of $G$.
\item[\emph{(b)}] If $p\in\pi_0$, and either $|\pi_0|\geq 3$ or $G$ has abelian Hall $\pi_0\smallsetminus\{p\}$-subgroups, then $\ze{H}=H\cap\ze{G}$, and $H$ is a quasi-Frobenius group with abelian kernel and complements.
\end{itemize}
\end{theorem}

\begin{proof}
Statement (a) follows from \cite[Theorem 5 (a)]{BF1} and \cite[Theorem 8]{BF3}, and statement (b) can be obtained from \cite[Corollary~10~(b), Theorem~12~(b)]{BF3}.
\end{proof}

\medskip

We point out that, by \cite[Theorem 4 (b)]{BF3}, if $p\in\pi_0$, then $|\pi_0|\geq 2$. Hence, as mentioned in the Introduction, it remains unsolved in the above theorem the situation where $\pi_0=\{p,q\}$ and a Sylow $q$-subgroup of $G$ is non-abelian.

A basic result when one works with $p$-regular conjugacy classes is the one below.

\begin{lemma}
\label{technical}
Suppose that $G$ is a $p$-separable group, for certain prime $p$, and that both $B=b^G$ and $C=c^G$ are non-central classes in $\conp$ of coprime lengths. Then:
\vspace*{-2mm}
\begin{itemize}
\setlength{\itemsep}{-1mm}
\item[\emph{(a)}] $BC=(bc)^G\in\operatorname{Con}(G_{p'})$ and $1\neq |BC|$ divides $|B|\cdot |C|$.
\item[\emph{(b)}] If $D\in\operatorname{Con}(G_{p'})$ is non-central, $d(D,C)\geq 3$ (allowing $d(D,C)=\infty$) and $|C|<|D|$, then $|DC|=|D|$, $\langle CC^{-1}\rangle\leqslant \langle DD^{-1}\rangle$ and $|\langle CC^{-1}\rangle|$ divides $|D|$.
\end{itemize}
\end{lemma}

\begin{proof}
This is partially \cite[Lemma 1]{LZ}.
\end{proof}

\medskip

Next we collect some features of non-central classes in $\operatorname{Con}(G_{p'})$ of maximal length.

\begin{proposition}
\label{prop}
Let $p$ be a prime, and suppose that $G$ is a $p$-separable group. Let $B_0\in\conp$ be a non-central conjugacy class of maximal length. Set $$ S=\langle D \; : \; D\in\conp \; \text{is non-central and} \; d(B_0,D)\geq 2\rangle. $$ Then the following properties hold.
\vspace*{-2mm}
\begin{itemize}
\setlength{\itemsep}{-1mm}
\item[\emph{(a)}] $S$ is an abelian normal $p'$-subgroup of $G$.
\item[\emph{(b)}] If $\ze{G}_{p'}$ denotes the $p$-complement of $\ze{G}$, then $\ze{G}_{p'}\leqslant S$ and $\pi(S/\ze{G}_{p'}) \subseteq \pi(B_0)$.
\item[\emph{(c)}] If $D=d^G\in\conp$ is a non-central conjugacy class such that $d(D,B_0)\geq 3$, then $\ce{G}{d}/S$ is a $\{p,q\}$-group, for some prime $q\in\pi(B_0)\cap\pi(o(d))$.
\item[\emph{(d)}] Let $D\in\conp$ be a non-central class such that $d(D,B_0)=3$. If $$D\mbox{ --- }A\mbox{ --- }C\mbox{ --- }B_0$$ is a path in $\Gamma_p(G)$, then $p\notin \pi(B_0)\cup\pi(D)$. Moreover, if $D'\in\conp$ is a non-central class such that $d(D',B_0)=3$, then $|D'|=|D|$.
\end{itemize}
\end{proposition}

\begin{proof}
Statements (a), (b) and (c) are proved in \cite[Proposition 1]{BF1} when $\Gamma_p(G)$ is connected, and in \cite[Theorem 4]{BF3} for the non-connected case. Statement (d) is also proved in \cite[Proposition 1]{BF1}: note that $d(D,B_0)=3$ forces the graph $\Gamma_p(G)$ to be connected, since otherwise the two connected components are complete subgraphs by \cite[Theorem 3]{BF1}.
\end{proof}

\medskip

The previous statements (c) and (d) will not actually occur due to Theorem~\ref{teoA}. 

As mentioned in the Introduction, the result below is a key fact in the proof of Theorem~\ref{teoA}.

\begin{theorem}[\text{\cite[Corollary C]{FMS}}]
\label{clave}
Let $G$ be a $\pi$-separable group, for a set of primes $\pi$. Then the following statements are pairwise equivalent:
\begin{enumerate}
\setlength{\itemsep}{-1mm}
	\item[\emph{(a)}] Each $\pi$-element $x\in G$ has class size either a $\pi$-number or a $\pi'$-number.
	
	\item[\emph{(b)}] Either $G=\rad{\pi}{G}\times\rad{\pi'}{G}$ or it has abelian Hall $\pi$-subgroups and its $\pi$-length is at most 1.
	
	\item[\emph{(c)}] For every $\pi$-element $x\in G$, either all $\abs{x^G}$ are $\pi$-numbers or they are all $\pi'$-numbers.
\end{enumerate}
\end{theorem}

We close this preliminary section with the so-called Wielandt's lemma and a technical result based on \cite[Lemma 6]{BF3}.

\begin{lemma}[\text{\cite[Lemma 1]{BK}}]
\label{wielandt}
Let $G$ be a group, and let $H\in\hall{\pi}{G}$ for a set of primes $\pi$. If $x\in H$ and $|x^G|$ is a $\pi$-number, then $x\in\rad{\pi}{G}$.
\end{lemma}

\begin{lemma}
\label{pi-sep}
Suppose that $G$ is a $p$-separable group, for a given prime $p$. Let $q$ be a prime. If $|x^G|$ is not divisible by $q$ for every $\{p,q\}'$-element $x\in G$, then $G$ is $\{p,q\}$-separable.
\end{lemma}

\begin{proof}
Observe that the hypotheses are inherited by quotients and normal subgroups of $G$. Let us argue by induction on $|G|$ to prove that $G$ is $q$-separable, and then the claim will follows directly. If $p\notin \pi(G)$, then each $q'$-element of $G$ has class length not divisible by $q$. In that case it is not difficult to show that $G=\rad{q}{G}\times\rad{q'}{G}$ (see for instance \cite[Proposition 2]{BF1}), so $G$ is $q$-separable. On the other hand, if $G$ is a $p$-group, then there is nothing to prove. Hence we may suppose that either $1<\rad{p}{G}<G$ or $1<\rad{p'}{G}<G$, and in both cases it follows by induction that $G$ is $q$-separable.
\end{proof}


\section{Proof of main results}

\begin{proof}[Proof of Theorem~\ref{teoA}]
The graph $\Gamma_p(G)$ is connected by hypothesis, so its diameter is at most 3 in virtue of \cite[Theorem 2]{BF1}. Let $B_0\in\conp$ be a non-central conjugacy class of maximal length. We aim to prove that $d(B_0,D)\leq 2$ for any non-central class $D\in\conp$. Arguing by contradiction, let us suppose that there exists a non-central conjugacy class $Y\in\conp$ such that $d(B_0,Y)=3$. Let $\pi_y=\pi(Y)$ and $\pi_0=\pi(B_0)$, so certainly $\pi_y\cap\pi_0=\emptyset$. By Proposition~\ref{prop} (a), we have that the subgroup $$S=\langle D \; : \; D\in\conp \; \text{is non-central and} \; d(B_0,D)\geq 2\rangle$$ is abelian, normal in $G$, and its order is not divisible by $p$. We proceed with a series of steps. We highlight that Steps 1 to 7 hold for any non-central $B_0=b^G\in\conp$ of maximal length, and by Proposition~\ref{prop} (d) they also hold for any non-central $Y=y^G\in\conp$ such that $d(B_0,Y)=3$.

\medskip

\noindent\textbf{\underline{Step 1.}} $G/\ze{G}$ is a $\pi_0\cup\pi_y\cup\{p\}$-group with $p\notin \pi_0\cup\pi_y$. 

\medskip

In virtue of Proposition~\ref{prop} (b) we have $\ze{G}_{p'}\leqslant S$ and $\pi(S/\ze{G}_{p'})\subseteq \pi_0$. Since $Y=y^G$ for certain $p'$-element $y\in G$, then by Proposition~\ref{prop} (c) we get that $\ce{G}{y}/S$ is a $\{p,q\}$-group, for some prime $q\in \pi_0$. Hence $|G:\ze{G}_{p'}|=|G:\ce{G}{y}|\cdot|\ce{G}{y}:S|\cdot|S:\ze{G}_{p'}|$ is a $\pi_y\cup\pi_0\cup\{p\}$-number, and in particular so is $|G:\ze{G}|$. Observe that $p\notin\pi_0\cup\pi_y$ by Proposition~\ref{prop} (d). 

\medskip

\noindent\textbf{\underline{Step 2.}} $G$ is $\pi_y$-separable, and there exists $T\in\hall{\pi_y}{G}$ abelian.

\medskip

We first claim that $G$ is $\pi_y$-separable. Take a $p'$-element $1\neq xS\in G/S$. We may assume that $x\notin S$ is a $p'$-element too, and so $d(x^G,B_0)\leq 1$ by the definition of $S$. Using Step 1, it follows that $|x^G|$ is a $\pi_0\cup \{p\}$-number, and so is $|(xS)^{G/S}|$. By \cite[Lemma 6]{BF3} we conclude that $G/S$ is a $\pi_0\cup\{p\}$-separable group and, since $S$ is abelian, then $G$ is $\pi_0\cup\{p\}$-separable. As $G/\ze{G}$ is a  $\pi_0\cup\pi_y\cup\{p\}$-group, then we may affirm that $G$ is $\pi_y$-separable.

Next we take any $\pi_y$-element $x\in G\smallsetminus \ze{G}_{p'}$. If $|x^G|$ is divisible by some prime in $\pi_y$, then necessarily $x\in S\smallsetminus \ze{G}_{p'}$, so $1\neq x\ze{G}_{p'}\in S/\ze{G}_{p'}$ which is a $\pi_0$-group by Proposition~\ref{prop} (b). But $o(x\ze{G}_{p'})$ divides $o(x)$, which is a $\pi_y$-number, so we get a contradiction. Thus $|x^G|$ is a $\pi_y'$-number for every $\pi_y$-element $x\in G$.  The existence of abelian Hall $\pi_y$-subgroups of $G$ now follows via Theorem~\ref{clave}.

\medskip

\noindent\textbf{\underline{Step 3.}} There exists a prime $q\in \pi_0$ such that $|y_q^G|=|Y|$, where $y_q\in S\smallsetminus\ze{G}_{p'}$ is the $q$-part of $y$.

\medskip

Since $Y$ is a generating class of $S$ by definition, then $y\in S\smallsetminus \ze{G}_{p'}$. Using the decomposition of $y$ as product of pairwise commuting elements of prime power order, we may certainly affirm that there exists a prime $q\in \pi(o(y))$, a $q$-element $y_q\in S$, and a $q'$-element $y_{q'}\in S$ such that $y=y_qy_{q'}=y_{q'}y_q$ with $y_q\notin \ze{G}$. Thus $1\neq |y_q^G|$ divides $|Y|$, and necessarily $d(y_q^G,B_0)=3$, so $|y_q^G|=|Y|$ by Proposition~\ref{prop} (d). In particular, as $y_q\in S\smallsetminus \ze{G}_{p'}$, Proposition~\ref{prop} (b) leads to $q\in \pi_0$. 

\medskip

\noindent\textbf{\underline{Step 4.}} $b\notin \ce{G}{y}$.

\medskip

By Step 3, it holds that $|y_q^G|=|Y|$ is a $\pi_y$-number, where $y_q$ is the $q$-part of $y$ for some prime $q\in \pi_0$. Note that $b=b_qb_{q'}=b_{q'}b_q$, so $|b_{q'}^G|$ divides $|B_0|$ and it is a $\pi_0$-number. By contradiction, let us suppose that $b\in\ce{G}{y}\leqslant\ce{G}{y_q}$, so $y_{q}\in\ce{G}{b}\leqslant\ce{G}{b_{q'}}$. It follows that $|(y_qb_{q'})^G|$ divides both $|y_q^G|$ and $|b_{q'}^G|$, which forces $b_{q'}\in\ze{G}$. Hence $|B|=|b_q^G|$, and there exists $T\in\hall{\pi_y}{G}$ such that $T\leqslant\ce{G}{b_q}$. Now $\ce{G}{b_qt}=\ce{G}{b_q}\cap\ce{G}{t}\leqslant \ce{G}{b_q}$ for every $t\in T$. Since $(b_qt)^G\in \conp$, and $|b_q^G|$ is of maximal length, then necessarily $\ce{G}{b_q}=\ce{G}{b_q}\cap\ce{G}{t}\leqslant\ce{G}{t}$, for every element $t\in T$. The assumption $yb=by$ yields that $y\in\ce{G}{b_q}\leqslant\ce{G}{t}$ for every element $t\in T$, so $T\leqslant \ce{G}{y}$, which is not possible.

\medskip

\noindent\textbf{\underline{Step 5.}} $|B_0|=|(b_{\pi_0} b_{\pi_y})^G|$, where $b_{\pi_0}\notin\ze{G}$ and $b_{\pi_y}\notin\ze{G}$ are the $\pi_0$-part and the $\pi_y$-part of $b$, respectively.

\medskip

Since $G/\ze{G}$ is a $\pi_0\cup\pi_y\cup\{p\}$-group by Step 1, then we may affirm that $|B_0|=|(b_{\pi_0} b_{\pi_y})^G|$. Certainly $b_{\pi_y}\notin\ze{G}$, since otherwise $|B_0|=|b_{\pi_0}^G|$ is of maximal length and therefore some conjugate of $b_{\pi_0}$ centralizes $y$, which contradicts Step 4. Next we aim to prove that $b_{\pi_0}\notin\ze{G}$. By contradiction, let us suppose the opposite, so $|B_0|=|b_{\pi_y}^G|$. Recall that $G$ is $\pi_y$-separable by Step 2, and $p$-separable by hypothesis. Using Step 1, we then have that $G$ is $\pi_0$-separable.

We claim that $G$ has abelian Hall $\pi_0$-subgroups. Let $z\in G$ be a $\pi_0$-element. Then $|z^G|$ is either a $\pi_0\cup\{p\}$-number or a $\pi_y\cup\{p\}$-number. In the first case, there exists $g\in G$ such that $b_{\pi_y}$ commutes with $z^g$. Hence $\ce{G}{b_{\pi_y}z^g}=\ce{G}{b_{\pi_y}}\cap\ce{G}{z^g}$, and by the maximality of $|b_{\pi_y}^G|$ we obtain $\ce{G}{b_{\pi_y}z^g}=\ce{G}{b_{\pi_y}}\leqslant\ce{G}{z^g}$. This implies that $|z^G|$ is not divisible by $p$, and thus it is a $\pi_0$-number. So the class size of each $\pi_0$-element of $G$ is either a $\pi_0$-number or a $\pi_y\cup\{p\}$-number. In virtue of Theorem~\ref{clave}, both cases cannot happen simultaneously, and by Step 3 we deduce that the second one actually occurs for all $\pi_0$-elements of $G$. Now the claim follows from Theorem~\ref{clave}.

Since $d(B_0,Y)=3$, there exist two vertices $A=a^G\in\conp$ and $C=c^G\in\conp$ that yield the following shortest path in the graph $\Gamma_p(G)$: $$B_0 \mbox{ --- }A\mbox{ --- }C\mbox{ --- }Y.$$ Note that $p$ is the unique prime that joins $A$ and $C$ in $\Gamma_p(G)$, by Step 1. Let us show that $a_{\pi_0}\in\ze{G}$ and so $|A|=|a_{\pi_y}^G|$, where $a_{\pi_0}$ and $a_{\pi_y}$ are the $\pi_0$-part and the $\pi_y$-part of $a$. Arguing by contradiction, if the $\pi_0$-part of $a$ is non-central in $G$, then $1\neq |a_{\pi_0}^G|$ is a $\pi_y\cup\{p\}$-number because the Hall $\pi_0$-subgroups of $G$ are abelian. But $|a_{\pi_0}^G|$ also divides the $\pi_0\cup\{p\}$-number $|A|$, so $|a_{\pi_0}^G|$ must be a non-trivial $p$-number. Now there exists some $g\in G$ such that $a_{\pi_0}^g\in\ce{G}{b_{\pi_y}}$. By the maximality of $|b_{\pi_y}^G|$, we deduce that $\ce{G}{b_{\pi_y}}\leqslant\ce{G}{a_{\pi_0}^g}$, which yields that $p$ divides $|b_{\pi_y}^G|=|B_0|$, a contradiction.

Next we claim that the $\pi_0$-elements of $\ce{G}{a}$ are central in $G$. Let $z\in\ce{G}{a}$ be a $\pi_0$-element. In particular, $z$ commutes with $a_{\pi_y}$, and as the $\pi_0$-part of $|a_{\pi_y}^G|=|A|$ is non-trivial, then necessarily $|z^G|$ is a $\pi_0\cup\{p\}$-number. In fact, $|z^G|$ is a $p$-number since the Hall $\pi_0$-subgroups of $G$ are abelian. Thus there exists $g\in G$ such that $z^g$ commutes with $b_{\pi_y}$. Since $|b_{\pi_y}^G|$ is of maximal size, arguing as above we get that $\ce{G}{b_{\pi_y}}\leqslant\ce{G}{z^g}$, so the $p$-number $|z^G|$ divides the $p'$-number $|b_{\pi_y}^G|$. This yields $|z^G|=1$. Hence $\ze{G}_{\pi_0}$, the Hall $\pi_0$-subgroup of $\ze{G}$, lies in $\hall{\pi_0}{\ce{G}{a}}$. An analogous reasonament leads to $\ze{G}_{\pi_0}\in\hall{\pi_0}{\ce{G}{b}}$. Therefore $$|A|_{\pi_0}=\dfrac{|G|_{\pi_0}}{|\ce{G}{a}|_{\pi_0}}=\dfrac{|G|_{\pi_0}}{|\ze{G}|_{\pi_0}}=\dfrac{|G|_{\pi_0}}{|\ce{G}{b}|_{\pi_0}}=|B_0|,$$ and since $p$ also divides $|A|$, then we get a contradiction with the maximality of $|B_0|$.

\medskip

\noindent\textbf{\underline{Step 6.}}  $b,y\in\rad{p'}{G}$ and $b_{\pi_0}\in\ze{\rad{p'}{G}}$.

\medskip

The first assertion directly follows from Step 1 and Lemma \ref{wielandt}. Set $N=\rad{p'}{G}$. Note that $b,y\notin \ze{N}$ by Step 4. Since $1\neq |b^N|$ and $1\neq |y^N|$ divide $|B_0|$ and $|Y|$, respectively, then in $\Gamma(N)$ we get that either $d(b^N, y^N)=3$ or $d(b^N,y^N)=\infty$. In the first case there exist two vertices $m^N$ and $n^N$ of $\Gamma(N)$ which yield a shortest path $b^N\mbox{ --- }m^N\mbox{ --- }n^N\mbox{ --- }y^N$. Hence there exists a prime $t\in \pi_0\cup\pi_y$ such that $t\in\pi(m^N)\cap\pi(n^N)\subseteq \pi(m^G)\cap\pi(n^G)$. This leads to a contradiction, since in this situation we would have $d(B_0,Y)\leq 2$ in $\Gamma_p(G)$. Hence $d(b^N,y^N)=\infty$, and then $\Gamma(N)$ is necessarily non-connected. It follows by Theorem~\ref{disconnected} that $N$ has two non-trivial class sizes, which must necessarily be $|y^N|$ and $|b^N|$. Let us suppose, arguing by contradiction, that $b_{\pi_0}\notin\ze{N}$. Since it is a $\pi_0$-element of $N$, $N$ has two non-trivial class sizes, and $N/\ze{N}$ is a Frobenius group with abelian inverse images in $N$ of the kernel and complements of $N/\ze{N}$, then we get that $|b_{\pi_0}^N|=|y^N|$. But $|b_{\pi_0}^N|$ divides $|B_0|$, which yields a contradiction.

\medskip

\noindent\textbf{\underline{Step 7.}}  $|b_{\pi_0}^G|=|(b_{\pi_0}y)^G|$.

\medskip

Recall that $Y$ and $b_{\pi_0}^G$ have coprime class lengths, so by Lemma~\ref{technical} (a) we get that $1\neq |(b_{\pi_0}y)^G|$ divides $|b_{\pi_0}^G|\cdot |Y|$. As this last product has prime divisors of both $\pi_0$ and $\pi_y$, then necessarily $|(b_{\pi_0}y)^G|$ must divide either $|Y|$ or $|b_{\pi_0}^G|$. In the first case we get $|(b_{\pi_0}y)^G|=|b_{\pi_0}^Gy^G|=|Y|>|b_{\pi_0}^G|$, and Lemma~\ref{technical} (b) yields $\langle b_{\pi_0}^G(b_{\pi_0}^G)^{-1}\rangle\leqslant \langle YY^{-1}\rangle$ and $|\langle b_{\pi_0}^G(b_{\pi_0}^G)^{-1}\rangle|$ divides $|Y|$. On the other hand, we can similarly argue with $Y$ and $b_{\pi_y}^G$, and we get again two cases. If $1\neq |(b_{\pi_y}y)^G|$ divides $|Y|$, then both class sizes must be equal, and so $b_{\pi_y}y\in S$. Since $y\in S$, then $b_{\pi_y}\in S\leqslant\ce{G}{y}$ because $S$ is abelian. Step 6 now yields $y\in\ce{G}{b_{\pi_0}}\cap\ce{G}{b_{\pi_y}}=\ce{G}{b}$, which is not possible by Step 4. Hence $|(b_{\pi_y}y)^G|=|b_{\pi_y}^G|>|Y|$ and, by Lemma~\ref{technical} (b), $|\langle YY^{-1}\rangle|$ divides $|b_{\pi_y}^G|$. It follows that $\langle b_{\pi_0}^G(b_{\pi_0}^G)^{-1}\rangle$ is both a $\pi_0$-group and a $\pi_y$-group, so $\langle b_{\pi_0}^G(b_{\pi_0}^G)^{-1}\rangle=1$ and $b_{\pi_0}\in\ze{G}$, which contradicts Step 5. Therefore we may conclude that $|(b_{\pi_0}y)^G|=|b_{\pi_0}^G|$, as desired.

\medskip

\noindent\textbf{\underline{Step 8.}} Final contradiction.

\medskip

Since $p\notin\pi_0\cup\pi_y$, there exists some $P\in\syl{p}{G}$ such that $P\leqslant\ce{G}{b}=\ce{G}{b_{\pi_0}}\cap\ce{G}{b_{\pi_y}}$, and we can choose a suitable $y\in G_{p'}$ with $d(B_0,y^G)=3$ such that $P\leqslant\ce{G}{y}$. Let $H\in\hall{p'}{\ce{G}{y}}$, so $\ce{G}{y}=PH$. Note that $G=T\ce{G}{y}=TPH$ for a suitable $T\in\hall{\pi_y}{G}$ such that $b_{\pi_y}\in T$, which is abelian by Step 2. If $z\in H\smallsetminus S$, then $|z^G|$ is a $\pi_0\cup\{p\}$-number by the definition of $S$, so $z\in\ce{H}{b_{\pi_y}^g}$ for some $g\in G=TPH$. We may then suppose that $g\in H$, and thus $$H=\displaystyle\bigcup_{g\in H} \left( \ce{H}{b_{\pi_y}}S\right) ^g.$$ We deduce that $H=\ce{H}{b_{\pi_y}}S$. Now by Step 7 there exists $g\in G$ such that $b_{\pi_y}^g\in \ce{G}{b_{\pi_0}y}$, and we may assume that $g\in S$. Clearly $b_{\pi_y}^g\in\ce{G}{b_{\pi_0}^g}=\ce{G}{b_{\pi_0}}$ because $b_{\pi_0}\in\ze{\rad{p'}{G}}\leqslant\ce{G}{S}$ by Step 6. It follows that $b_{\pi_y}^g\in \ce{G}{y}$, or equivalently $b_{\pi_y}\in \ce{G}{y^{g^{-1}}}=\ce{G}{y}$. This leads to $b=b_{\pi_0}b_{\pi_y}\in\ce{G}{y}$, which is not possible by Step 4.
\end{proof}

\medskip

Undoubtedly, the upper bound provided in Theorem~\ref{teoA} is best possible since in the ordinary graph $\Gamma(G)$ is known to be sharp. We can use this result to alternatively prove \cite[Theorem 2]{BF1}, that we restate below.

\begin{corollary}
\label{cor_diam}
Suppose that $G$ is a $p$-separable group. If $\Gamma_p(G)$ is connected, then its diameter is at most three.
\end{corollary}

\begin{proof}
By contradiction, let us suppose that there exist two non-central classes $A$ and $B$ in $\conp$ such that $d(A,B)=4$. By Theorem~\ref{teoA}, we necessarily deduce that $d(A,B_0)=2=d(B,B_0)$, where $B_0\in\conp$ is a non-central conjugacy class of maximal length. We may suppose, without loss of generality, that $|A|>|B|$, so $|\langle BB^{-1}\rangle|$ divides $|A|$ by Lemma~\ref{technical} (b). On the other hand, $B_0B\in\conp$ is non-central by Lemma~\ref{technical} (a), and thus $|B_0|=|B_0B|$. Since $B_0B$ and $B^{-1}$ are classes of coprime sizes, then $B_0BB^{-1}\in\conp$ again by Lemma~\ref{technical} (a). Note that $B_0\subseteq B_0BB^{-1}$, so it follows $B_0= B_0BB^{-1}$. Hence $B_0\langle B B^{-1}\rangle = B_0$ and therefore $|\langle BB^{-1}\rangle|$ divides $|B_0|$. But this is a contradiction since $A$ and $B_0$ have coprime lengths and $B$ is non-central.
\end{proof}

\medskip

\begin{proof}[Proof of Theorem~\ref{teoB}]
Let $G$ be a $p$-separable group. Consider that $d(x^G, y^G)=3$, where $x^G, y^G\in\conp$ are non-central. Set $\pi_x=\pi(x^G)$, $\pi_y=\pi(y^G)$ and $\pi=\pi_x\cup \pi_y$. Let $H$ be a $p$-complement of $G$ such that $x,y\in H$ up to conjugation. We will argue by induction on $|G|$ in order to show that, if $p$ divides $\min(|x^G|,|y^G|)$, then $G$ is $p$-nilpotent, $H=\rad{\pi}{H}\times\rad{\pi'}{H}$ with $\rad{\pi}{H}$ quasi-Frobenius with abelian kernel and complements, each complement of $H$ is centralised by some Sylow $p$-subgroup of $G$, and $\ze{H}=H\cap\ze{G}$.

Note that, in the decompositions of $x$ and $y$ as product of pairwise commuting elements of prime power order, we may suppose that each element is non-central: this is because if $x=x_1x_2$ with $x_1\notin\ze{G}$ and $x_2\in\ze{G}$, then $|x^G|=|(x_1x_2)^G|=|x_1^Gx_2|=|x_1^G|$ and $\pi(x^G)=\pi(x_1^G)$.

\medskip

\noindent\textbf{\underline{Step 1.}} If $xy=yx$, then $\pi(o(x))=\pi(o(y))=\{q\}$ for some prime $q\neq p$, and $q$ divides $\operatorname{max}(|x^G|,|y^G|)$. 

\medskip

Suppose first that $xy=yx$, and $x_{q},y_{r}\notin\ze{G}$ for two primes $q\neq r$ different from $p$. Since $|x_q^G|$ is a $\pi_x$-number, $|y_r^G|$ a $\pi_y$-number, and $x_qy_r=y_rx_q$, then $(x_qy_r)^G\in\conp$ and it has length divisible by primes lying in $\pi_x$ and in $\pi_y$. This implies the contradiction $d(x^G,y^G)\leq 2$.

Next we show that $q\in\operatorname{max}(|x^G|,|y^G|)$. We may certainly assume $|x^G|>|y^G|$, so by Lemma~\ref{technical} we get that $\langle y^G(y^G)^{-1}\rangle\leqslant \langle x^G(x^G)^{-1}\rangle$ and $|\langle y^G(y^G)^{-1}\rangle|$ divides $|x^G|$. But $x^G\subseteq \ce{G}{y}$ because $G=\ce{G}{x}\ce{G}{y}$, so $\langle y^G(y^G)^{-1}\rangle\leqslant\langle x^G(x^G)^{-1}\rangle\leqslant\ce{G}{y}$. In particular, taking some $g\in G$ such that $z=y^g\neq y$, we obtain $1\neq w=zy^{-1}\in \langle y^G(y^G)^{-1}\rangle\leqslant\ce{G}{y}$. Hence $z\in\ce{G}{y^{-1}}$ and therefore $w$ is a non-trivial $q$-element lying in $\langle y^G(y^G)^{-1}\rangle$. Since we already know that the order of this subgroup divides $|x^G|$, we have then proved that $q\in\pi_x$.

\medskip

\noindent\textbf{\underline{Step 2.}} If $xy\neq yx$, then either $\pi(o(x))\subseteq \pi_y$ and $\pi(o(y))\subseteq \pi_x$, or $x=x_qx_{q'}$ and $y=y_q$ with $q\in\pi_x$ and $\emptyset\neq \pi(o(x_{q'}))\subseteq \pi_y$.

\medskip

Note that $\pi(o(y))\cap\pi_x\neq \emptyset$, since otherwise, as $|x^G|=|\ce{G}{y}:\ce{G}{x}\cap\ce{G}{y}|$ is a $\pi_x$-number and $y\in\ze{\ce{G}{y}}$, then $xy=yx$ which cannot occur. Hence there exists a prime $q\in\pi(o(y))\cap\pi_x$. Analogously, there exists a prime $r\in\pi(o(x))\cap \pi_y$. Indeed, we can similarly deduce that $\pi(o(y))\smallsetminus\{r\}\subseteq\pi_x$, and analogously $\pi(o(x))\smallsetminus\{q\}\subseteq\pi_y$. It follows that $\pi(o(x))\cap\pi(o(y))\subseteq \{q,r\}$, because $\pi_x\cap\pi_y=\emptyset$.

Now if $\pi(o(x))\cap\pi(o(y))=\emptyset$, then $\pi(o(x))\subseteq \pi_y$ and $\pi(o(y))\subseteq \pi_x$. If $\pi(o(x))\cap\pi(o(y))=\{q\}$, then as $q\in\pi_x$ it follows that $q\notin\pi_y$, and so there exists a conjugate of $x_q$ that commutes with $y$. But $G=\ce{G}{x_q}\ce{G}{y}$ since they have coprime class sizes, so $x_qy=yx_q$. Consequently $y_{q'}\in\ze{G}$, and $y$ can be assumed to be a $q$-element, with $q\in\pi_x$ and $\pi(o(x_{q'}))\subseteq \pi_y$. In this situation observe that $\pi(o(x_{q'})\neq \emptyset$: otherwise both $x$ and $y$ are $q$-elements, and as $q\notin \pi_y$, then it can be easily proved that $x$ commutes with $y$, which is not possible. Finally, if $\pi(o(x))\cap\pi(o(y))=\{q,r\}$, then we can similarly deduce that $y$ can be chosen to be a $q$-element, which contradicts the fact $r\in\pi(o(y))$.

We remark that, as a consequence of the previous two steps, both $x$ and $y$ are $\pi$-elements hereafter.

\medskip

\noindent\textbf{\underline{Step 3.}} $G/\ze{G}$ is not a $\pi$-group.

\medskip

The opposite would imply that in a path of length three $x^G\mbox{ --- }a^G\mbox{ --- }c^G\mbox{ --- }y^G$ the unique prime that joins $a^G$ and $c^G$ could be $p$. Since $p\in \pi=\pi_x\cup \pi_y$ by hypotheses, then necessarily $d(x^G,y^G)\leq 2$ which is against our assumptions.

\medskip

\noindent\textbf{\underline{Step 4.}} $G$ is $\pi$-separable.

\medskip

Let us denote $K=\ce{G}{x}\cap\ce{G}{y}$ and $\sigma=\pi(o(x))\cup\pi(o(y))\cup\{p\}\subseteq\pi$. We claim that each $\sigma'$-element of $K$ has class size in $G$ a $\pi'$-number. Let $z\in K$ be a $\sigma'$-element; note that $z$ must exist, since otherwise $K$ would be a $\sigma$-group, and as $|G:K|=|x^G|\cdot|y^G|$ it would follow that $|G|=|G:K|\cdot|K|$ is a $\pi$-number, which is not possible due to Step 3. Consequently $z$ commutes with both $x$ and $y$, and it has coprime order with them, so $|z^G|$ divides both $|(zx)^G|$ and $|(zy)^G|$. But these two numbers are divisible by $|x^G|$ and $|y^G|$, respectively, which leads to $\pi\cap\pi(z^G)=\emptyset$ due to the assumption $d(x^G,y^G)=3$.

Set $T=\pmb{\operatorname{O}}^{\sigma}(K)=\langle g\in K \: : \: \pi(o(g))\subseteq\sigma'\rangle$. We have shown above that $T\neq 1$. Note also that $T\unlhd G$, since for every generator $g\in T$ it holds $(|G:K|,|g^G|)=1$, so $G=K\ce{G}{g}$ and $g^G=g^K\subseteq T$.

Let us show that $T$ is $\sigma$-separable. If $xy=yx$, then Step 1 implies that $\sigma=\{p,q\}$ for some prime $q\neq p$. Thus every $\sigma'$-element $\alpha\in T\leqslant K$ verifies that $|\alpha^G|$ is a $\pi'$-number, and in particular $|\alpha^T|$ is not divisible by $q$. Now Lemma~\ref{pi-sep} yields the claim. 

On the other hand, if $xy\neq yx$, then by Step 2 we need to discuss two situations: either $x$ and $y$ have coprime orders or they have not. In the former case, $\pi(o(x))\subseteq\pi_y$ and $\pi(o(y))\subseteq \pi_x$ by Step 2. Take any prime $s\in\pi(o(x))\cup\pi(o(y))$ and any $\{p,s\}'$-element $\alpha\in T\leqslant K$. If $s\in\pi(o(x))$, then $x_s$ and $\alpha$ have coprime orders, and $x_s\alpha=\alpha x_s$, which forces $|(x_s\alpha)^G|$ to be a $\pi_y'$-number. As $\pi(o(x))\subseteq \pi_y$, then $|\alpha^G|$ is not divisible by $s$, and $|\alpha^T|$ is not either. Utilising the same argument, one can show that the conclusion also holds if $s\in\pi(o(y))$. Therefore, $|\alpha^T|$ is not divisible by $s$ for every $\{p,s\}'$-element $\alpha\in T\leqslant K$, and for every prime $s\in\sigma$ different from $p$. In virtue of Lemma~\ref{pi-sep} we get the $\sigma$-separability of $T$. 

Finally, in the last case where $x$ and $y$ do not have coprime orders, by Step 2 we get $\pi(o(y))=\{q\}\subseteq \pi_x$ and $x=x_qx_{q'}$ with $\emptyset\neq\pi(o(x_{q'}))\subseteq \pi_y$. It is possible to similarly deduce that every $\{p,q\}'$-element $\alpha\in T$ has $|\alpha^T|$ not divisible by $q$, and that for each prime $s\in\pi(o(x_{q'}))$ and for every $\{p,s\}'$-element $\alpha\in T$ its class size in $T$ is not divisible by $s$. Again Lemma~\ref{pi-sep} yields the desired claim.

Hence, we have deduced in all cases that $T$ is $\sigma$-separable, and so it is $\sigma'$-separable. Recall that, by the above paragraphs, every $\sigma'$-element of $T$ has class size in $T$ a $\sigma'$-number. Theorem~\ref{clave} leads to $T=\rad{\sigma}{T}\times\rad{\sigma'}{T}$, and actually $T=\rad{\sigma'}{T}$ since it is generated by $\sigma'$-elements by definition. In particular every element of $T$ has class size not divisible by $s$, for every prime $s\in \pi$, so each Sylow $s$-subgroup of $T$ is central and $T=\rad{\pi}{T}\times \rad{\pi'}{T}$ with $\rad{\pi}{T}$ abelian.

Observe that $\rad{\pi'}{T}\in\hall{\pi'}{G}$ since $|G:T|=|G:K|\cdot|K:T|$, which is the product of a $\pi$-number and a $\sigma$-number, respectively, with $\sigma\subseteq\pi$. In particular $G$ is $\pi$-separable.

\medskip

\noindent\textbf{\underline{Step 5.}} If $B_0\in\conp$ is a class of maximal cardinality, and $\operatorname{min}(|x^G|,|y^G|)=|y^G|$, then $d(B_0,y^G)=2$. 

\medskip

Let us suppose the opposite, so $d(x^G,B_0)=2$ by hypotheses and Theorem~\ref{teoA}. Set $X=x^G$ and $Y=y^G$. Applying Lemma~\ref{technical}, we get $XB_0\in\conp$, and necessarily $|XB_0|=|B_0|$. Hence $(|X^{-1}|, |XB_0|)=1$, so again Lemma~\ref{technical} yields $X^{-1}XB_0\in\conp$. Since $B_0\subseteq X^{-1}XB_0$, then by maximality $X^{-1}XB_0=B_0$. It follows $\langle X^{-1}X\rangle B_0=B_0$, which implies that $|\langle X^{-1}X\rangle|$ divides $|B_0|$. Moreover, since $|X|>|Y|$ and $d(X,Y)=3$, in virtue of Lemma~\ref{technical} we get $\langle Y^{-1}Y\rangle\leqslant\langle X^{-1}X\rangle$ and $1\neq |\langle Y^{-1}Y\rangle|$ divides $|X|$, which is a $\pi_x$-number. Thus we have obtained the contradiction $(|B_0|, |X|)\neq 1$.

\medskip

\noindent\textbf{\underline{Step 6.}} The conclusion.

\medskip

Note that every $\pi'$-element $g\in G$ lies in $\rad{\pi'}{T}\leqslant K$, and so $|g^G|$ is a $\pi'$-number by the first paragraph of Step 4. Since we have proved that $G$ is $\pi$-separable (and so $\pi'$-separable), in virtue of Theorem~\ref{clave} we deduce that $G=\rad{\pi}{G}\times \rad{\pi'}{G}$, and thus $H=\rad{\pi}{H}\times \rad{\pi'}{H}$. Now $x^G=x^{\rad{\pi}{G}}$ and $y^G=y^{\rad{\pi}{G}}$, so $p\in\pi(x^{\rad{\pi}{G}})\cup \pi(y^{\rad{\pi}{G}})$. Clearly $d(x^{\rad{\pi}{G}}, y^{\rad{\pi}{G}})\geq 3$ because otherwise, since $\rad{\pi}{G}$ is normal in $G$, then we may induce in $\Gamma_p(G)$ a path of length strictly smaller than 3 from $\Gamma_p(\rad{\pi}{G})$, and this is not possible. 

Now we distinguish two cases. If $d(x^{\rad{\pi}{G}}, y^{\rad{\pi}{G}})= 3$, as $\rad{\pi}{G}<G$ by Step 3, then we may apply induction to this normal subgroup to deduce that $\rad{\pi}{H}\in\hall{p'}{\rad{\pi}{G}}$ is a quasi-Frobenius group with abelian kernel and complements, $\rad{\pi}{H}$ is normal in $\rad{\pi}{G}$, $\ze{\rad{\pi}{H}}=\rad{\pi}{H}\cap\ze{\rad{\pi}{G}}$, and each complement of $\rad{\pi}{H}$ is centralised by some Sylow $p$-subgroup of $\rad{\pi}{G}$. Since $|G:\rad{\pi}{G}|$ is a $p'$-number, we get the desired conclusion.

Finally, let us suppose that $d(x^{\rad{\pi}{G}}, y^{\rad{\pi}{G}})= \infty$, so $\Gamma_p(\rad{\pi}{G})$ has two complete connected components by \cite[Theorem 3]{BF1}. Clearly, a class $B_0\in\conp$ of maximal length can be decomposed as $B_0=B_1B_2$, where $B_1\in\text{Con}(\rad{\pi}{G}_{p'})$ and $B_2\in\text{Con}(\rad{\pi'}{G})$ are of maximal length, respectively. In particular, $|B_0|=|B_1|\cdot |B_2|$. Let us suppose w.l.o.g. that $\operatorname{min}(|x^G|,|y^G|)=|y^G|$. Now Step 5 leads to $d(B_0, y^G)=2$, and thus $d(B_1, y^{\rad{\pi}{G}})=\infty$. Hence $y^{\rad{\pi}{G}}$ and $B_1$ belong to different connected components of $\Gamma_p(\rad{\pi}{G})$. Since $y^G=y^{\rad{\pi}{G}}$ and $p\in\pi(y^G)$ by hypotheses, then it is enough to apply Theorem~\ref{disconnected-p-reg} (a) to $\rad{\pi}{G}$ in order to finish the proof.
\end{proof}

\medskip

\begin{example}
\label{example}
It is not difficult to obtain groups that satisfy the hypotheses of Theorem~\ref{teoB}.  In fact, it is enough to consider a $p'$-group $H=\rad{\pi}{H}\times\rad{\pi'}{H}$ with $\rad{\pi}{H}$ quasi-Frobenius with abelian kernel and complements, and then a $p$-group $P$ that only acts on the kernel of $\rad{\pi}{H}$. 

For instance, take $p=2$ and $G$ the direct product of a cyclic group of order $7$ acted frobeniusly by a cyclic group of order $6$, and an extraspecial group of order $5^3$ and exponent $5$, i.e. $G\cong (C_7 \rtimes C_6) \times ((C_5\times C_5)\rtimes C_5)$. There are classes $x^G,y^G\in\conp$ with sizes $6$ and $7$, respectively, and $d(x^G,y^G)=3$ in $\Gamma_p(G)$ with $p=2$ dividing $\min(|x^G|,|y^G|)$. 

On the other hand, we highlight that $G$ may not be $p$-nilpotent when $p$ divides $\max(|x^G|,|y^G|)$. This is the case, for example, of a direct product $G$ of a dihedral group of order $42$ and an extraspecial group of order $5^3$ and exponent $5$, with $p=3$. Here there are $p$-regular classes with lengths $2$ and $21$, and their distance in $\Gamma_p(G)$ is $3$.

Finally observe that, by considering the action of $P$ on $\rad{\pi'}{H}$ (or vice versa), one can build groups such that $d(x^G,y^G)=3$ in $\Gamma_p(G)$ for certain classes $x^G,y^G\in\conp$ with $p\notin \pi(x^G)\cup\pi(y^G)$.
\end{example}

\medskip

\begin{remark}
\label{remark}
Let $G$ be a $p$-separable group such that $\Gamma_p(G)$ is non-connected. Let $B_0\in\conp$ be a non-central class of maximal cardinality, and let $\pi_0$ be the set of prime divisors of the lengths of all conjugacy classes lying in the same connected component as $B_0$. As stated in Theorem~\ref{disconnected-p-reg}, it is known that $H$ is quasi-Frobenius with abelian kernel and complements provided that $p\notin \pi_0$. However, the situation $p\in\pi_0$ is not yet entirely solved, since the case where $\pi_0=\{p,q\}$ and a Sylow $q$-subgroup of $G$ is non-abelian is still open.

The previous issue is closely related to the behaviour of $\Gamma_p(G)$ when it has diameter three and $p$ divides $\operatorname{max}(|x^G|,|y^G|)$, since the proof of Theorem~\ref{teoB} can be reduced to the case that $\Gamma_p(G)$ is non-connected with $p\in\pi_0$. Therefore, a positive answer to the open problem mentioned in the above paragraph directly provides a proof of Conjecture~\ref{conC} in the general case $p\in\pi(x^G)\cup \pi(y^G)$.
\end{remark}


\end{document}